\documentclass[reqno]{amsart}
\usepackage{amsfonts, amsmath, amsthm, amssymb, latexsym, xfrac, mathrsfs, braket, color}

\makeatletter
\@namedef{subjclassname@2020}{\textup{2020} Mathematics Subject Classification}
\makeatother

\theoremstyle{plain}
\newtheorem{theorem}{Theorem}[section]

\newtheorem{lemma}[theorem]{Lemma}

\theoremstyle{definition}

\theoremstyle{remark}

\numberwithin{equation}{section}

\title[On the representing measures of Dunkl's intertwining operator]
{Remarks on a proof of the absolute continuity of the representing measures for Dunkl's intertwining operator     }
\author{Dominik Brennecken and Margit R\"osler} 
\address{Institut f\"ur Mathematik, Universit\"at Paderborn, Warburger Str. 100, D-33098 Paderborn, Germany}
\email{brennecken.dominik@web.de, roesler@math.upb.de}

\thanks{The authors were supported by the  German Research Foundation (DFG), via the grant SFB-TRR  358/1 2023-491392403.}

\begin{document}
\date{\today}

\begin{abstract}
In the published paper \cite{T10},   K. Trim\`eche presented a proof 
of the statement that for suitable nonnegative 
multiplicities and regular arguments, the representing measures of Dunkl's intertwining operator and its dual are absolutely continuous with respect to Lebesgue measure. In this note, we argue that the essential proofs of this paper are not correct.

\end{abstract}

\maketitle

\section{Background}

Suppose that $R$ is a reduced, not necessarily crystallographic root system in a finite-dimensional Euclidean space $(\mathfrak a, \langle\,.\,,\,.\,\rangle)$.  It is not required to be spanning, i.e. $\text{span}_{\mathbb  R} R$ may be a proper subspace of $\mathfrak a.$ Let $W$ be the associated finite reflection group and  $k: R\to  [0,\infty)$  a non-negative, $W$-invariant multiplicity function. Let further $T_\xi(k) \> (\xi\in \mathfrak a)$ denote the rational Dunkl operators associated with $R$ and $k$ as introduced in \cite{D89}.
Dunkl's intertwining operator (\cite{D91}) is characterized as the unique linear isomorphism  $V_k$ of 
  the space $\mathbb C[\mathfrak a]$, the polynomial functions on $\mathfrak a$, which preserves the degree of homogeneity and satisfies 
  $$ T_\xi(k) V_k = V_k \partial_\xi \>\>\text{for all } \xi \in \mathfrak a, \>V_k(1) =1.$$
  It is known from \cite{R99} that for each $x\in \mathfrak a, $ there exists a unique probability measure $\mu_x^k\in M^1(\mathfrak a)$ such that
  \begin{equation}\label{intrep} V_kp(x) = \int_{\mathfrak a} p(\xi)d\mu_x^k(\xi) \quad \text{for all } p\in \mathbb C[\mathfrak a]. \end{equation}
   The measure $\mu_x^k\,$ is compactly supported in the convex hull of the $W$-orbit of $x.$ 
  By formula \eqref{intrep}, $V_k$ naturally extends to larger function spaces. For example,  $V_k: C(\mathfrak a) \to C(\mathfrak a),$   see \cite{T01}.
  The importance of $V_k$ is due to its close relation to the  Dunkl kernel, which replaces the exponential function in classical  Fourier analysis. In fact, 
  $$ E_k(x,z) = \int_{\mathfrak a} e^{\langle\xi,z\rangle} d\mu_x^k(\xi) \quad \text{for all } z\in \mathfrak 
  a_{\mathbb C}.$$
  This constitutes an abstract  generalization of the Harish-Chandra integral representation for the spherical functions of a Riemannian symmetric space of Euclidean type, which can be identified, for crystallographic $R$ and particular values of $k$, 
  with  the Bessel functions $\,J_k(x, z) =  \frac{1}{|W|} \sum_{w\in W}E_k(wx, z), z \in  \mathfrak a_{\mathbb C}.$
  
  Let $\,\frak a_{reg}= \{x\in \frak a: \langle\alpha, x\rangle\not=0 \text{ for all }\alpha \in R\}$
  denote the regular elements of $\frak a.$ 
  Motivated by well-known results in the group cases mentioned above (\cite[Ch.IV, Prop. 4.10 and Ch.II, Thm. 10.11]{H00}), the following is a long-standing conjecture, formulated for instance in \cite{RJ02}. 
   
  \bigskip\noindent
  \textbf{Conjecture A.} Let $k \geq 0$ and suppose that the set 
  $R^\prime:= \{\alpha\in R: k(\alpha)>0\}$ spans $\frak a$. 
  Then for each $x \in \frak a_{reg}$, the measure $\mu_x^k\,$ is absolutely continuous with respect to Lebesgue measure on $\mathfrak a.$ 
  
  \medskip

 As already explained in \cite{RJ02},  the condition on $R^\prime$ is indeed necessary:  Let $V^\prime:=\text{span}_\mathbb R(R^\prime)$ and $V^{\prime\prime}:= (V^\prime)^\perp.$ For $x\in \mathfrak a$ write $x=x^\prime+ x^{\prime\prime}$ accordingly. Then $\mu_x^k$ is supported in  $x^{\prime\prime} +V^{\prime} $, and is therefore not absolutely 
 continuous unless $V^\prime =\mathfrak a.$ 
 
 \medskip

 Recently in \cite{L26},  Conjecture A was established at least for $k>\frac{1}{2}, $ as a consequence of deep uniform bounds on the Dunkl kernel up to the boundary of a Weyl chamber. But to our knowledge, it is still open in the general case.
 
 \medskip
 
  In this note, we comment on an earlier, erroneous proof of Conjecture A in \cite{T10}. There the space $\frak a$ is always the Euclidean $\mathbb R^d$ with its standard inner product.
  
 \section{The arguments concerning \cite{T10}}
 
  In the following, $m$ denotes  the Lebesgue measure on $\mathbb R^d$. For a complex Borel measure $\mu$ on $\mathbb R^d$, write $\mu = \mu^a + \mu^s$ for its Lebesgue decomposition w.r.t. $m,$ where $\mu^a$ and $\mu^s$ are the  absolutely continuous and singular part of $\mu$, respectively.   
   
 \bigskip\noindent
 (1)
 In \cite{T10}, the condition on $R^\prime$ in Conjecture A is  actually also discussed in Proposition 2.5, but nevertheless, it is lateron neither stated as a condition nor is it used in the proofs. 
Instead, before entering the decisive Section 4.1, the author announces on p. 210 to prove that for $d\geq 2$, the measure $\mu_x^k$ is absolutely continuous 
 under the condition 
 $$ k(\alpha)>0 \text{ for all } \alpha\in R_+.$$
 This condition is too weak, and also a specification of necessary conditions concerning the regularity of $x$ is missing. In fact, the only condition appearing in Section 4.1 is the  even weaker requirement
 $$\gamma = \sum_{\alpha\in R_+} k(\alpha) >0.$$
 This already reveals that the statements of Section 4.1. cannot be fully correct.
 
 \bigskip\noindent
 (2) There are various issues in Section 4.1. of \cite{T10} which are not correct as stated or have an incorrect proof, but may probably still be repaired.  For instance, estimate (4.7) is not correct, as taking the limit $\gamma\uparrow 1$ reveals; this estimate would need some modification. Further in formula (4.6), the case $y_0=0$ has to be excluded.
 
 A more serious error is contained in Theorem 4.6. Here 
 for $r\in (0,\infty)$ and positive  $f\in C_c(\mathbb R^d),$  the author considers the positive measure defined by
 \begin{equation}\label{lambda_r^f}\lambda_r^f(E) =\int_{S^{d-1}} f(r\xi) \mu_{r\xi}(E)\omega_k(r\xi)d\sigma(\xi), \quad E\subseteq\mathbb R^d \text{ a Borel set.} \end{equation}
 This is a weighted spherical mean of the measures $\mu_x$, taken  over the sphere $S(0,r)$ of radius $r$ centered at $0.$ 
  Theorem 4.6(i) states, for the case $0<\gamma<1$, that $\lambda_r^f$ is absolutely continuous  on $\mathbb R^d\setminus S(0,r)$. From this the author then concludes  that $\lambda_r^f$ has a Lebesgue density on $\mathbb R^d.$ This conclusion is invalid. For instance, it is not excluded that $\lambda_r^f$ might be concentrated on the sphere  $S(0,r)$ which 
 has Lebesgue measure zero.  
 It may be true that $\lambda_r^f$ is indeed absolutely continuous, even under the week condition $\gamma>0,$ but this would require a different proof.

 \bigskip\noindent
 (3) The proof of Proposition 4.7 contains a major error: 
 The author again only requires $\gamma>0$. Considering the decomposition 
 $$ \lambda_r^f(E) =\int_{S^{d-1}} f(r\xi) \mu_{r\xi}^a(E)\omega_k(r\xi)d\sigma(\xi)\,+\, \int_{S^{d-1}} f(r\xi) \mu_{r\xi}^s(E)\omega_k(r\xi)d\sigma(\xi),$$
 he concludes directly from his previous statement $\lambda_r^f \ll m\,$ that 
 here the second summand must be zero, i.e. 
 $$ m_r^f(E):= \int_{S^{d-1}} f(r\xi) \mu_{r\xi}^s(E)\omega_k(r\xi)d\sigma(\xi)=0$$ 
 for all  Borel sets $E\subseteq \mathbb R^d.$ But this conclusion is invalid. Assuming measurability of the mappings $\xi\mapsto \mu_{r\xi}^a (E)$ and $\xi\mapsto \mu_{r\xi}^s(E),$ this conclusion would be correct only if the measure $m_r^f$ were known to be singular with respect to $m.$ But the singularity of the measures $\mu_{r\xi}^s$ alone does not imply that their spherical means $m_r^f$ are also singular.
 Therefore the proof of Proposition 4.7 breaks down. 
 
 \medskip\noindent
 (4) 
 Proposition 4.7 and Theorem 4.6 (ii)  are essential ingredients in the proof of Theorem 4.8, which claims absolute continuity of the representing measures of the dual intertwiner $^tV_k.$ Theorem 4.8 in turn is essential towards Theorem 4.11, which states absolute continuity of the measures $\mu_x$ for $x\in \mathbb R^d_{reg}$, still under the insufficient requirement $\gamma>0$ (which however goes unmentioned). Therefore the error described in (3) is fatal for the remaining parts of the paper. 
 
 \medskip\noindent
 (5) It should also be mentioned that \cite{T10} is often careless about measurability matters or statements which hold only almost everywhere. For example, the measurability of $\xi\mapsto \mu_{r\xi}(E)$ needed in \eqref{lambda_r^f} is not proven for general Borel sets $E,$  but only for open balls. As it may be of independent interest, we shall provide a proof for this fact below. 
 Further, in the proof of Theorem 4.8 the author assumes that $\lambda_r^f \ll m$ with density $\mathcal H^f(r,\,.\,)$ and then claims that
 $$ \lim_{a\to 0} \frac{\lambda_r^f(B(y_0,a))}{m(B(y_0,a))}\,=\, \mathcal H^f(r, y_0)$$
 for all $y_0\in \mathbb R^d.$ However, the Lebesgue differentiation theorem  gives this limit only for almost all $y_0\in \mathbb R^d.$ This restriction would then apply to formula (4.19) for $^tV_k(f)(y_0).$

   \section{Measurability of $x \mapsto \mu_x^k(E)$ } 
   
  \noindent
   We return to the general setting of Section 1, assuming only  $k \geq 0.$
   
 \begin{lemma} 
 Let $E\subseteq  \mathfrak a$ be a Borel set. Then the mapping
 $\, x\mapsto \mu_x^k(E)$ is measurable on $\mathfrak a.$ 
  \end{lemma}
 
 \begin{proof}  This is standard. First consider an arbitrary non-empty, half-open rectangle $Q = \prod_{i=1}^d [a_i, b_i)\subset \mathfrak a,$ with coordinates with respect to some fixed orthonormal basis of $\mathfrak a.$ Choose a sequence $(\varphi_n)_{n\in \mathbb N} \subset C_c(\mathfrak a)$ which converges monotonically increasing to the characteristic function $1_{Q}.$ Then, by the monotone convergence theorem, 
 $$ \mu_x^k(Q) = \lim_{n\to \infty} \int_{\mathfrak a}\varphi_n d\mu_x^k \, = \, \lim_{n\to \infty} (V_k\varphi_n)(x),$$
 which is measurable in $x$ as a pointwise limit of continuous functions. Consider the system
 $ \mathcal E$ consisting of all non-empty, half-open rectangles $Q\subset \mathfrak a.$ The system $\mathcal E$ is stable under intersections and generates the $\sigma$-algebra $\mathcal B(\mathfrak a)$ of Borel subsets of $\mathfrak a$. Therefore the Dynkin system 
 generated by $\mathcal E$ equals $\mathcal B(\mathfrak a).$ On the other hand, 
 $$ \mathcal D= \{E \in \mathcal B(\mathfrak a)\, \vert\, x\mapsto \mu_x^k(E)  \text{ is measurable} \} $$
 is a Dynkin system containing $\mathcal E.$ Therefore $\mathcal D = \mathcal B(\mathfrak a).$ 
  \end{proof}

\end{document}